\documentclass[11pt]{article}
\usepackage[hmargin=1in,vmargin=1in]{geometry}
\usepackage{amsmath}
\usepackage{amsthm}
\usepackage{graphicx}
\usepackage{placeins}
\usepackage{ esint }
\usepackage[maxbibnames=99, style=numeric]{biblatex}
\bibliography{references.bib}
\usepackage{url}
\usepackage{breakurl}

\usepackage{amsthm,amsfonts,amssymb,euscript,verbatim,bbm}

\usepackage{graphics}
\usepackage{enumerate}

\usepackage{color}
\usepackage{xcolor}

\newtheorem{theorem}{Theorem}
\newtheorem{proposition}[theorem]{Proposition}

\newtheorem{remark}[theorem]{Remark}

\newcommand{\eqdef}{\overset{\mbox{\tiny{def}}}{=}}
\def\vh {\hat{\pel}}

\def\ls {\lesssim}
\def\vp{\varphi}

\def\rd {\partial}

\def\f {\frac}

\def \R {\mathbb{R}}

\newcommand{\pel}{p}
\def\vh {\hat{\pel}}

\def\ls {\lesssim}

\def\rd {\partial}

\def\f {\frac}

\newcommand{\ba}{\begin{equation}}
	\newcommand{\ea}{\end{equation}}

\newcommand{\bea}{\begin{eqnarray}}
	\newcommand{\eea}{\end{eqnarray}}

\def\beaa{\begin{eqnarray*}}
	\def\eeaa{\end{eqnarray*}}

\title{Bounds on the Aspect Ratio of the Momentum Support of a 2D Collisionless Plasma}
\date{}
\author{ {Matthew Hernandez}$\,^{1}\;,\quad$
{Neel Patel}$\,^{1}\quad$ and $\quad$
{Elena Salguero}$\,^{2}$ \vspace{.5cm} \\
\footnotesize{$\,^{1}\;$ \textsc{University of Maine}} \\
{\footnotesize 5752 Neville Hall, Room 237
Orono, ME 04469} \vspace{.3cm} \\
\footnotesize{$\,^{2}\;$ \textsc{Max Planck Institute for Mathematics in the Sciences}} \\
{\footnotesize Inselstrasse 22, 04103 Leipzig,
Germany} \vspace{.3cm} \\
\footnotesize{\hspace{-0.7cm}
Email addresses: $\,\;$\ttfamily{matthew.b.hernandez@maine.edu}},
\footnotesize{$\,$ $\;$\ttfamily{neel.patel@maine.edu}},
\footnotesize{$\,$ $\;$\ttfamily{elena.salguero@mis.mpg.de}}
\vspace{.2cm}
}

\begin{document}

\maketitle

\begin{abstract}
The relativistic Vlasov-Maxwell system is a kinetic model for collisionless plasmas. For the two-dimensional model, global well-posedness of this model is known and was proven by deriving global bounds on the momentum support of the particle density function. In this paper, we prove bounds on the magnitude of the momentum support in one direction depending on the magnitude of the support in the corresponding orthogonal direction.
\end{abstract}

\section{Introduction}

Plasmas are completely ionized gases. When a plasma is assumed to be collisionless, the dominant forces are the electric and magnetic forces created by the charged particles and their currents. The relativistic Vlasov-Maxwell (rVM) system models the particle density and electromagnetic fields of collisionless plasmas with fast moving particles. In this paper, the traditional problem of a mono-charged plasma is considered in two dimensions. The plasma is described by the non-negative particle density function $$f:\mathbb{R}_{t}\times\mathbb{R}^{2}_{x}\times\mathbb{R}_{p}^{2}\rightarrow [0,\infty).$$
Here, $x\in\mathbb{R}^{2}$ is the spatial variable and $p \in\mathbb{R}^{2}$ is the relativistic momentum or velocity variable. We define the normalized momentum $\hat{p} = p/p_{0}$, where $p_{0} = \sqrt{1+|p|^{2}}$.  The charge and current densities of the plasma are given by 
 \begin{align*}
     \rho(t,x) \eqdef 4 \pi \int_{\R^2} f(t,x,p) dp, \quad   j(t,x) \eqdef 4 \pi \int_{\R^2} \hat{p} f(t,x,p) dp,
 \end{align*}
 respectively. We denote the electric and magnetic fields by $$E = (E_1, E_2):\mathbb{R}_{t}\times\mathbb{R}^{2}_{x} \rightarrow \mathbb{R}^{2}$$ and $$B :\mathbb{R}_{t}\times\mathbb{R}^{2}_{x} \rightarrow \mathbb{R},$$ respectively. Note that by the symmetry of the two-dimensional system, the magnetic field given by the scalar function above represents a vector field $(0,0,B)$ that is orthogonal to the two-dimensional plane.

In two dimensions, the rVM system is given by:
	\bea
	& &\rd_t f+\vh\cdot\nabla_x f+  (E_1 + \hat{p}_2 B, E_2 -\hat{p}_1 B)\cdot \nabla_\pel f = 0,\label{vlasov}\\
	& &\rd_t E= \nabla^{\perp}_{x} B - j,\quad \rd_t B= \nabla^{\perp}_{x} \cdot E,\label{maxwell}\\
	& &\nabla_x\cdot E=\rho ,\label{constraints}
	\eea
where $\nabla_{x} =  (\partial_{x_{1}}, \partial_{x_{2}})$, $\nabla^{\perp}_{x} = (\partial_{x_{2}}, -\partial_{x_{1}})$ and $\nabla_p =  (\partial_{p_{1}}, \partial_{p_{2}})$.

 The Vlasov equation \eqref{vlasov} represents the kinetic movement of particles, whereby particles move at velocities $\hat{p} = p/p_{0}$ for relativistic velocity $p\in\mathbb{R}^{2}$. The acceleration is due to the two-dimensional electromagnetic force
 \begin{align}\label{Kforce}
 K(t,x,p) \eqdef (E_1 + \hat{p}_2 B, E_2 -\hat{p}_1 B).
 \end{align}
 The electromagnetic forces in turn are given by \eqref{maxwell} and \eqref{constraints}. This system is inherently a coupled relativistic kinetic and wave equation system describing the evolution of a collisionless plasma.

 The global well-posedness of the rVM system in two dimensions has been resolved \cite{GS1998}. In the first part of \cite{GS1998}, global well-posedness of a $C^{1}$ solution is established with the criteria that the momentum support \eqref{suppf} of the density function remains bounded for all finite times. In this paper, we will denote the momentum support of the density function by
 \begin{equation}\label{suppf}
 \Omega(t) \eqdef \{p \in \mathbb{R}^{2} \ |  \ f(s,x,p) \neq 0 \text{ for some } 0\leq s \leq t, x\in\mathbb{R}^{2}\}.
 \end{equation}
 In the second part of \cite{GS1998}, it is proven that the momentum support indeed does not become unbounded in finite time. To do so, the momentum is bounded by the acceleration due to the electromagnetic forces along characteristic curves \eqref{momentumaccelerationbound}. In turn, the electromagnetic forces are bounded by the size of the momentum support. We state the estimate from \cite{GS1998} on the electromagnetic force here:
 \begin{align}\label{GSfieldestimate}
\sup_{x\in\mathbb{R}^{2}}|E(t,x)| + \sup_{x\in\mathbb{R}^{2}}|B(t,x)| \leq \mathcal{C}(t)\tilde{P}(t)\log(\tilde{P}(t)),
 \end{align}
where $\mathcal{C}(t)$ is a continuous function in time and the function $\tilde{P}(t)$ denotes the size of the momentum support $\Omega(t)$:
\begin{align}\label{originalP}
\tilde{P}(t) = \sup_{p\in\Omega(s)} |p|.
\end{align}
Note that \eqref{suppf} implies that $\tilde{P}(t)$ in non-decreasing. By a Gronwall inequality type argument, the momentum support is then shown to be bounded by an exponentially growing function in time and global well-posedness of the two-dimensional rVM system is established. It should also be noted that the two and one-half dimensional rVM system, in which the momentum domain is three-dimensional but the spatial domain is two-dimensional, is also globally well-posed \cite{GS1997}, \cite{LukStrain2016}.

On the other hand, in the three dimensional problem, global well-posedness remains an open problem. Again, if the momentum support remains bounded for all time, the rVM system is globally well-posed: see \cite{GlasseyStrauss1986} and later alternative proofs \cite{BouchutGolsePallard2003}, \cite{KlainermanStaffilani2002}. Following this result, several works have proven criteria for global well-posedness: \cite{GlasseyStrauss1987}, \cite{GlasseyStrauss1989}, \cite{Kunze2015}, \cite{LukStrain2014}, \cite{Pallard2015}, \cite{Patel2018},  \cite{AlfonsoIllner2010}. Even if the momentum support is unbounded, given certain moments are bounded for all time, global well-posedness is known \cite{LukStrain2016}. Global existence is proved in the cylindrically-symmetric problem in \cite{WangCylindrical}. For studies of global solutions in cases of spherical symmetry and near-spherical symmetry, we refer the reader to \cite{Horst1990} and \cite{Rein1990}.
 
 Let's return back to the rVM system in lower dimensions. In the $1\frac{1}{2}$-dimensional rVM system, the spatial variable is one-dimensional and the momentum variable is two-dimensional. For that setting, \cite{GlasseyPankavichSchaeffer2010} establishes bounds on large time behavior of the momentum. They prove that in the non-spatial direction, the momentum is bounded proportionally to $1 + \sqrt{t-|x|+C_0}$. In this paper, we investigate large-time quantitative estimates in the full two-dimensional rVM setting. Specifically, we study the aspect ratio of the momentum support, i.e. its relative size in each direction: \textit{Given a bound of the momentum support in one direction, what is a bound on the momentum support in the orthogonal direction?}

 \subsubsection*{Outline of the paper}

 In the rest of this section, we set up the structure and notation of the paper and state the main result. In Section 2, estimates on the electromagnetic force are computed. Finally, in Section 3, these bounds are used in conjunction with the structure set up in Section 1 to prove the main result.

 \subsection{Setup and Main Result}

We will consider initial particle density $f(0,x,p)$ to be compactly supported in both $x$ and $p$ variables. Then, in the $e_2$ direction, we define
\begin{align}\label{P2define}
P_{2}(t) =\sup_{p\in\Omega(s)} |p_2|.
\end{align}
Hence, $P_2(t)$ is a nondecreasing function by \eqref{suppf}. Next, let $0 \leq w < 1$ be fixed. Then, the momentum support $\Omega(t)$ is contained in
\begin{align}\label{rectangle}
\Omega(t) \subset [-P(t),P(t)]\times[-P_{2}(t),P_{2}(t)]
\end{align}
where $P(t)$ is also a nondecreasing function, given by
\begin{align}\label{P}
P(t) = \tilde{P}(t) + P_{2}(t)^{\frac{1}{w}}.
\end{align}
In \eqref{P}, $\tilde{P}(t)$ is given by \eqref{originalP} and $w$ is determined later in the paper. We consider the relative growth of the upper bound on the momentum support $P(t)$ depending on the growth only in one direction $P_2(t)$. Note that 
\begin{align}\label{ratiolimit}
0 \leq \frac{P_{2}(t)}{P(t)^{w+\delta}} \rightarrow 0
\end{align}
as $t\rightarrow \infty$ assuming that $P_{2}(t) \rightarrow \infty$ as $t\rightarrow \infty$.

The support of the density function can be determined by the method of characteristics. The density $f(t,x,p)$ is constant along characteristic curves $X(s) = X(s;t,x,p)$ and $V(s) = V(s;t,x,p)$ given by the transport (Vlasov) equation \eqref{vlasov}:
\begin{align*}
\dot{X}(s) &= \hat{V}(s),\\
\dot{V}(s) &=  K(s,X(s), V(s))
\end{align*}
where $K(s,X(s),V(s))$ is given by \eqref{Kforce} and the initial data is
\begin{align*}
X(t;t,x,p) = x \quad \text{and} \quad V(t;t,x,p) = p,
\end{align*}
where 
\begin{align*}
    \hat{V} \eqdef \frac{V}{\sqrt{1+|V|^2}}.
\end{align*}
 One immediate consequence is that the density function is uniformly bounded
 \begin{equation}\label{linfinity}
 \|f\|_{L^{\infty}_{x,p}}(t) = \|f_0\|_{L^{\infty}_{x,p}}.
 \end{equation} 
Furthermore, we can compute the magnitude of the momentum characteristic curve as follows:
\begin{align*}
\frac{1}{2}\frac{d}{dt}|V(t)|^{2} &= V(t)\cdot \dot{V}(t)\\
&= V(t) \cdot K(t,X(t),V(t))\\
&\leq |V(t)||K(t,X(t), V(t))|,
\end{align*}
which gives
\begin{align*}
|V(t; t,x,p)| - |V(0;t,x,p)|\leq \int_{0}^{t}|K(s,X(s;t,x,p), V(s;t,x,p))|ds.
\end{align*}
Thus,
\begin{align}\label{momentumaccelerationbound}
|p| \leq |V(0;t,x,p)| + \int_{0}^{t} |K(s,X(s;t,x,p),V(s;t,x,p))| ds.
\end{align}
Hence,
\begin{align*}
\sup_{0\leq s \leq t} \sup_{p\in \Omega(s)} |p| \leq \sup_{0\leq s \leq t}\sup_{p\in \Omega(s)}  |V(0;t,x,p)| +  \sup_{0\leq s \leq t}\sup_{p\in \Omega(s)} \int_{0}^{s} |K(s',X(s';t,x,p),V(s;t,x,p))| ds'.
\end{align*}
Then, using the definition given in \eqref{P}, we obtain that
\begin{align}\label{Ptbound}
P(t) \leq P_{2}(t)^{\frac{1}{w}} + P(0) +  \int_{0}^{t} \sup_{x\in\mathbb{R}^{3}}\sup_{p\in\Omega(s)}|K(s,x,p)| ds.
\end{align}
From \eqref{Ptbound}, the bound on the quantity $P(t)$ requires control on the electromagnetic field $K$. We decompose $K(t,x)$ following \cite{GS1998}. The electromagnetic field $K$ can be decomposed into three components and bounded as follows: Let us denote $|K(t,x,p)| \leq |K(t,x)| \eqdef |E| + |B|$ for convenience as $\hat{p} \leq 1$ in \eqref{Kforce}. Then, $$|K(t,x)| \leq  c_{0} (1+ K_{T}(t,x)+ K_{S,1}(t,x)+K_{S,2}(t,x)),$$ in which $c_{0}$ is a fixed constant dependent on the initial value $K(0,x,p)$ and
 \begin{equation}\label{KT}
 K_{T}(t,x) = \int_{0}^{t}\int_{|y-x|\leq t-s}\int_{\mathbb{R}^{2}}\frac{f(s,y,p) \ dp \ dy \ ds}{p_{0}^{2}(1+\hat{p}\cdot\xi)^{\frac{3}{2}}(t-s)\sqrt{(t-s)^2-|y-x|^2}},
 \end{equation}

\begin{equation}\label{KS1}
 K_{S,1}(t,x) = \int_{0}^{t}\int_{|y-x|\leq t-s}\int_{\mathbb{R}^{2}}\frac{|K(s,y)|f(s,y,p) \ dp \ dy \ ds}{p_{0}\sqrt{(t-s)^2-|y-x|^2}},
 \end{equation}
and
\begin{equation}\label{KS2}
 K_{S,2}(t,x) = \int_{0}^{t}\int_{|y-x|\leq t-s}\int_{\mathbb{R}^{2}}\frac{K_{g}(s,y,\omega)f(s,y,p) \ dp \ dy \ ds}{p_{0}(1+\hat{p}\cdot\xi)\sqrt{(t-s)^2-|y-x|^2}},
 \end{equation}
 where $$\xi = (y-x)/(t-s)$$ and $$K_{g}^{2} = |E\cdot\omega|^{2}+|B+\omega_1 E_2 - \omega_2 E_1|^{2}$$ for $$\omega=(y-x)/|y-x|.$$ 
Importantly, the following conservation law holds for the $K_g$ quantity in \eqref{KS2}:
\begin{equation}\label{eq:consKg}
    \frac{1}{4} \int_{C_{t,x}} K_g^2 d \sigma +4 \pi \int_{C_{t,x}} \int_{\mathbb{R}^2} p_0 (1+\hat{p} \cdot \omega) f \ dp \ d \sigma \leq C,
\end{equation}
where $C_{t,x}$ represents the backward null cone emanating from $(t,x)$, which is defined as 
\begin{equation}
    C_{t,x} = \{ (s,y) \in \mathbb{R} \times \mathbb{R}^2 \, | \, 0 \leq s \leq t, \, t-s = |y-x| \}.
\end{equation}
The density function and electromagnetic fields also satisfy the following conservation laws:
\begin{align}\label{totalmass}
    \int_{\mathbb{R}^{2}}\int_{\mathbb{R}^{2}} f(t,x,p) dp dx = \text{constant}
\end{align}
and
\begin{align}\label{L2conservation}
\frac{1}{2}\int_{\mathbb{R}^{2}} |E(t,x)|^{2} + |B(t,x)|^{2} dx + 4\pi \int_{\mathbb{R}^{2}}\int_{\mathbb{R}^{2}} p_0 f(t,x,p) dp dx = \text{constant}.
\end{align}
The first conservation law \eqref{totalmass} means that the total mass of the plasma is constant in time. For example, these conservation laws are written for the two,  three, and $2\frac{1}{2}$-dimensional rVM systems in \cite{LukStrain2016}. 
The remainder of this paper is dedicated to bounding the components \eqref{KT}, \eqref{KS1} and \eqref{KS2} and thereby proving our main result:
\begin{theorem}
Consider initial data $f_0 \in C^{1}$ and $E_0, B_0 \in C^{2}$ such that $f_0$ is compactly supported in $x$ and $p$ and $E_0, B_0 \in L^{2}(\mathbb{R}^{2})\cap W^{2,\infty}(\mathbb{R}^{2})$. Then there exists a $T \geq 0$ such that for any $t > T$, the momentum support $\Omega(t)$ of the density function $f(t,x,p)$ is contained in the region
\begin{align*}
    [- c t^{8} P_{2}(t)^{3} \log(t P_{2}(t)))^{3}, c t^{8} P_{2}(t)^{3} \log(t P_{2}(t)))^{3}] \times [-P_{2}(t),P_{2}(t)]
\end{align*}
where $c$ is a fixed constant.
\end{theorem}

\begin{remark}
Throughout the proof, we use that $t \lesssim P(t)$, i.e., the growth of $P(t)$ is larger than linear. Note that otherwise, the theorem would trivially hold because $P_{2}(t) \geq P_{2}(0) > 0$ by \eqref{P2define} and \eqref{totalmass}.

In this remark and in the proof of the theorem, we use the notation $Y \lesssim Z$ to denote that $Y \leq c Z$ for some time independent constant $c > 0$.
\end{remark}

\begin{remark}
The regularity of the initial data $f_{0} \in C^{1}$ is in the variables $x$ and $p$ and $E_{0}, B_{0} \in C^{2}$ is in the variable $x$. This regularity ensures a global $C^{1}$ solution $f(t,x,p)$ in the variables $t,x,p$, as shown in \cite{GS1998}.
\end{remark}

\section{Bounds on the Electromagnetic Forces}

In this section, we find bounds on each component \eqref{KT}, \eqref{KS1}, \eqref{KS2} of the electromagnetic force. For the estimates on $K_T$ and $K_{S,1}$, we pick a fixed $T > 1$ large enough so that $P(T) \geq 2$.

\begin{proposition}
The following estimate holds for all $x\in\mathbb{R}^{2}$ and $t\geq T$:
\begin{equation}\label{KTfinalbound}
 	K_T(t,x)\lesssim t P_{2}(t)\log(P(t)).
 \end{equation}
 \end{proposition}
 \begin{proof}
 Performing the change of variables $\xi = (y-x)/(t-s)$ in \eqref{KT}, using \eqref{suppf} and \eqref{linfinity}, and then changing the order of integration, we obtain
\begin{align}
 K_{T}(t,x) & = \int_{0}^{t}\int_{|\xi|\leq 1}\int_{\mathbb{R}^{2}}\frac{f(s,x+\xi(t-s),p) \ dp \ d\xi \ ds}{p_{0}^{2}(1+\hat{p}\cdot\xi)^{\frac{3}{2}}\sqrt{1-|\xi|^2}}\nonumber\\
 &\lesssim \int_{0}^{t}\int_{\Omega(s)}\int_{|\xi|\leq 1}\frac{ d\xi \ dp \ ds}{p_{0}^{2}(1+\hat{p}\cdot\xi)^{\frac{3}{2}}\sqrt{1-|\xi|^2}}.\label{KT_bound1}
\end{align}
Now we note that for a given $\hat{p}$, we may rotate the integration variable $\xi$ so that $\hat{p}$ points in the $\xi_{1}$ direction. That gives rise to the following identity:
 \begin{align*}
 	\int_{|\xi|\leq 1}\frac{d\xi}{(1+\hat p\cdot\xi)^{3/2}\sqrt{1-|\xi|^2}}
 	&=\int_{|\xi|\leq 1}\frac{d\xi}{(1+|\hat p|\xi_1)^{3/2}\sqrt{1-|\xi|^2}}\\
 &=  \int^1_{-1}\frac{1}{(1+|\hat p|\xi_1)^{3/2}}\int^{\sqrt{1-\xi^2_1}}_{-\sqrt{1-\xi^2_1}}\frac{1}{\sqrt{1-\xi^2_1-\xi^2_2}}d\xi_2d\xi_1\\
 	&=\pi\int^1_{-1}\frac{1}{(1+|\hat p|\xi_1)^{3/2}}d\xi_1\\
  &= 2\pi \frac{1}{|\hat p|}\left[\frac{1}{\sqrt{1-|\hat p|}}-\frac{1}{\sqrt{1+|\hat p|}}\right].
 \end{align*}
In the above calculation of the $\xi_2$ integral, we used the following calculation that holds for any $a \in (0,1)$:
\begin{align*}
\int_{-a}^{a} \frac{1}{\sqrt{a^{2}-x^{2}}} dx &= \int_{-1}^{1} \frac{1}{\sqrt{a^{2}-a^{2}y^{2}}} a dy = \int_{-1}^{1} \frac{1}{\sqrt{1-y^{2}}}  dy = \arcsin(1) - \arcsin(-1) = \pi.
\end{align*}
Furthermore, we have that
\begin{align*}
\frac{1}{\sqrt{1-|\hat p|}}-\frac{1}{\sqrt{1+|\hat p|}} &= \frac{\sqrt{1+|\hat p|}-\sqrt{1-|\hat p|}}{\sqrt{1-|\hat p|^2}}\\
&= p_{0} \left(\sqrt{1+|\hat p|}-\sqrt{1-|\hat p|}\right)\\
&= 2 p_{0}\frac{|\hat p|}{\sqrt{1+|\hat p|}+\sqrt{1-|\hat p|}}\\
&\leq 2 p_{0} |\hat{p}|.
\end{align*}
Therefore,
\begin{align*}
 	\int_{|\xi|\leq 1}\frac{d\xi}{(1+\hat p\cdot\xi)^{3/2}\sqrt{1-|\xi|^2}} &\leq 4\pi p_{0}.
\end{align*}

 Recall that by \eqref{rectangle}, the momentum support $\Omega(t)$ of $f(t,x,p)$ is contained in the rectangle $[-P(t),P(t)]\times [-P_2(t), P_2(t)]$. Thus, for $0\leq s \leq t$,
 \begin{align}
\int_{\Omega(s)}\frac{1}{p_0}dp &\lesssim \int^{P_2(s)}_{-P_2(s)}\int^{P(s)}_{-P(s)}\frac{1}{(1+p_{1}^{2}+p_{2}^{2})^{\frac{1}{2}}}dp_1 dp_2 \nonumber\\
&\lesssim \int^{P_2(s)}_{0}\int^{P(s)}_{0}\frac{1}{1+p_{1}}dp_1 dp_2  \nonumber\\
&= \int^{P_2(s)}_{0} \log(1+P(s))dp_2  \nonumber\\
&\lesssim P_{2}(s)\log(1+P(s)) \nonumber\\
& \lesssim P_{2}(t)\log(P(t))\label{usefulestimatektks1}
 \end{align}
 where we used that $P(t)$ and $P_2(t)$ are non-decreasing and $P(t) \geq P(T) \geq 2$.
 Hence, \eqref{KT_bound1} becomes
 \begin{align*}
 K_T(t,x) \lesssim  t P_{2}(t)\log(P(t)).
 \end{align*}
 
 \end{proof}

\subsection{Estimates on $K_{S,1}$}

\begin{proposition}
    The following estimate holds for all $x\in \mathbb{R}^2$ and $t \geq T$:
    \begin{align}\label{KS1finalbound}
K_{S,1}(t,x) &\lesssim t \sqrt{\log(t)} P_{2}(t)\log(P(t)) + t P_{2}(t)\log(P(t))\|K\|_{L^{\infty}([0,t]\times\mathbb{R}^{2})}^{\mu} 
\end{align}
for any choice $0 < \mu \leq 1$. The implicit constant depends on $\mu$.
\end{proposition}

Before the proof of this proposition, we would like to note that this bound on $K_{S,1}$ is not sharp. However, because the bound on $K_{S,2}$ done later in the paper is greater in large time, we write the bound on $K_{S,1}$ in this way for simplicity.

\begin{proof}
The term $K_{S,1}$ is given by \eqref{KS1}:
\begin{align*}
    K_{S,1}(t,x) = \int_{0}^{t}\int_{|y-x|\leq t-s}\int_{\mathbb{R}^{2}}\frac{|K(s,y)|f(s,y,p) \ dp \ dy \ ds}{p_{0}\sqrt{(t-s)^2-|y-x|^2}}.
\end{align*}
We will subdivide the region of integration into three parts:
\begin{align*}
   & 0\leq s \leq t-\delta \ \text{     and  }  \ |y-x|\leq t-s-\varepsilon,\\
  &  0 \leq s \leq t-\delta \  \text{  and  }    \ t-s-\varepsilon\leq |y-x|\leq t-s,\\
  &  t-\delta\leq s \leq t. &
\end{align*}
Note that we need $\delta \geq  \varepsilon$ to make sense of the first two cases.

These parts represent different regions of the solid lightcone of height $t$ and base radius $t$. In particular, the third region represents the corner of the cone, while the first and second region are far from the corner and they represent the interior of the cone and the layer closer to the boundary, respectively.

In the first case, we use the conservation law \eqref{L2conservation} and estimate \eqref{usefulestimatektks1} to obtain
\begin{align*}
K_{S',1}(t,x) &\eqdef \int_{0}^{t-\delta}\int_{|y-x|\leq t-s-\varepsilon}\int_{\mathbb{R}^{2}}\frac{|K(s,y)|f(s,y,p) \ dp \ dy \ ds}{p_{0}\sqrt{(t-s)^2-|y-x|^2}}\\
&\leq \int_{0}^{t-\delta}\|K(s,\cdot)\|_{L^{2}(\mathbb{R}^{2})} \Big(\int_{|y-x|\leq t-s-\varepsilon}\Big(\int_{\mathbb{R}^{2}}\frac{f(s,y,p)}{p_{0}} \ dp\Big)^{2}\frac{1}{{(t-s)^2-|y-x|^2}} \ dy \Big)^{\frac{1}{2}} \ ds\\
&\lesssim \int_{0}^{t-\delta}\Big(\int_{|y-x|\leq t-s-\varepsilon}\Big(P_{2}(t)\log(P(t))\Big)^{2}\frac{1}{{(t-s)^2-|y-x|^2}} \ dy \Big)^{\frac{1}{2}} \ ds\\
&= P_{2}(t)\log(P(t)) \int_{0}^{t-\delta}\Big(\int_{0}^{t-s-\varepsilon}\int_{0}^{2\pi}\frac{1}{{(t-s)^2-r^2}} \ r \ d\theta \ dr\Big)^{\frac{1}{2}} \ ds\\
&\lesssim   P_{2}(t)\log(P(t))\int_{0}^{t-\delta}\Big(\int_{2\varepsilon(t-s)-\varepsilon^2}^{(t-s)^{2}}\frac{1}{u} \ du\Big)^{\frac{1}{2}} \ ds\\
&=  P_{2}(t)\log(P(t)) \int_{0}^{t-\delta}\sqrt{2\log(t-s)-\log(\varepsilon)-\log(2(t-s)-\varepsilon)}\ ds\\
&\lesssim t P_{2}(t)\log(P(t))\sqrt{|\log(\varepsilon)|},
\end{align*}
where we make the choice that $\varepsilon\leq \delta = t^{-1} < 1$ for $t> 1$ and hence, we have $|\log(t-s)|\leq |\log(\delta)|\leq |\log(\varepsilon)|$ for $0 \leq s \leq t-\delta$. The other term only matters in the estimate if $\log(2(t-s)-\varepsilon) < 1$ in which case, since $2(t-s)-\varepsilon \geq 2\delta - \varepsilon \geq \varepsilon$, we again get $|\log(2(t-s)-\varepsilon)| \leq |\log(\varepsilon)|$.

Next, we use again the conservation law \eqref{L2conservation} and estimate \eqref{usefulestimatektks1}, together with Hölder inequality for $1 \leq q<2$ and $q'>2$ with $\frac{1}{q}+\frac{1}{q'} = 1$ to find 
\begin{align*}
 K_{S'',1}(t,x) &\eqdef \int_{0}^{t-\delta}\int_{t-s-\varepsilon \leq |y-x|\leq t-s}\int_{\mathbb{R}^{2}}\frac{|K(s,y)|f(s,y,p) \ dp \ dy \ ds}{p_{0}\sqrt{(t-s)^2-|y-x|^2}}\\
 &\leq \int_{0}^{t-\delta}\|K(s,\cdot)\|_{L^{q'}(\mathbb{R}^{2})} \Big(\int_{t-s-\varepsilon\leq |y-x|\leq t-s}\Big(\int_{\mathbb{R}^{2}}\frac{f(s,y,p)}{p_{0}} \ dp\Big)^{q}\frac{1}{{((t-s)^2-|y-x|)^{\frac{q}{2}}}} \ dy \Big)^{\frac{1}{q}} \ ds\\
 &\lesssim  P_{2}(t)\log(P(t)) \int_{0}^{t-\delta}\|K(s,\cdot)\|_{L^{\infty}(\mathbb{R}^{2})}^{\mu}\Big(\int_{t-s-\varepsilon}^{t-s}\frac{1}{{((t-s)^2-r^2)^{\frac{q}{2}}}} r \ dr \Big)^{\frac{1}{q}} \ ds\\
  &\lesssim P_{2}(t)\log(P(t)) \int_{0}^{t-\delta}\|K(s,\cdot)\|_{L^{\infty}(\mathbb{R}^{2})}^{\mu} (t-s)^{\frac{1}{q}-\frac{1}{2}} \varepsilon^{\frac{1}{q}-\frac{1}{2}}\ ds,
\end{align*}
where we used interpolation 
\begin{align*}
  \|K(s,\cdot)\|_{L^{q'}(\mathbb{R}^{2})} \lesssim \|K(s,\cdot)\|_{L^{2}(\mathbb{R}^{2})}^{1-\mu} \|K(s,\cdot)\|_{L^{\infty}(\mathbb{R}^{2})}^{\mu}  
\end{align*}
for $\mu = \frac{2}{q}-1$. This implies that $\mu$ takes values in the range $0 < \mu \leq 1$. Thus,
\begin{align*}
K_{S'',1}(t,x) &\lesssim t^{\frac{1}{q}+\frac{1}{2}} P_{2}(t)\log(P(t))\varepsilon^{\frac{1}{q}-\frac{1}{2}}\|K\|_{L^{\infty}([0,t]\times\mathbb{R}^{2})}^{\mu},
\end{align*}
where the implicit constant in the inequality depends on $\mu$.
Finally, we can apply the same argument as for $K_{S'',1}$ to get
\begin{align*}
 K_{S''',1}(t,x) &\eqdef \int_{t-\delta}^{t}\int_{0 \leq |y-x|\leq t-s}\int_{\mathbb{R}^{2}}\frac{|K(s,y)|f(s,y,p) \ dp \ dy \ ds}{p_{0}\sqrt{(t-s)^2-|y-x|^2}}\\
 &\lesssim P_{2}(t)\log(P(t))\int_{t-\delta}^{t}\|K(s,\cdot)\|_{L^{\infty}(\mathbb{R}^{2})}^{\mu}\Big(\int_{0}^{t-s}\frac{1}{{((t-s)^2-r^2)^{\frac{q}{2}}}} r \ dr \Big)^{\frac{1}{q}} \ ds\\
 &\lesssim P_{2}(t)\log(P(t))\int_{t-\delta}^{t}\|K(s,\cdot)\|_{L^{\infty}(\mathbb{R}^{2})}^{\mu} (t-s)^{\frac{2}{q}-1} \ ds\\
 &\lesssim P_{2}(t)\log(P(t))\|K\|_{L^{\infty}([0,t]\times\mathbb{R}^{2})}^{\mu} \delta^{\frac{2}{q}}.
\end{align*}
Picking $\varepsilon = \delta = t^{-1}$, we obtain
\begin{align*}
K_{S,1} &\lesssim t \sqrt{\log(t)} P_{2}(t)\log(P(t)) + t P_{2}(t)\log(P(t))\|K\|_{L^{\infty}([0,t]\times\mathbb{R}^{2})}^{\mu} 
\end{align*}
for any value $0 < \mu \leq 1$, where again, the implicit constant depends on $\mu$.
\end{proof}

\subsection{Estimates on $K_{S,2}$}

\begin{proposition}
    The following estimate holds for all $x\in \mathbb{R}^2$ and all $t > T$ for $T$ sufficiently large:
    \begin{align}\label{KS2finalbound}
K_{S,2}(t,x) &\lesssim   t P_2(t)^{\frac{3}{4}} P(t)^{\frac{3}{4}}  \log(P(t))^{\frac{3}{4}}.
\end{align}
\end{proposition}

The rest of this section is devoted to prove the previous proposition. The term $K_{S,2}$ needs further decomposition. As in \cite{GS1998}, define
\begin{align}\label{sigmaS}
    \sigma_S(t,y,\xi)=\int_{\mathbb{R}^{2}} \frac{f(t,y,p)}{p_0(1+\hat p\cdot \xi)} \ dp,
\end{align}
where $t\geq 0$ and $y,\xi \in \mathbb{R}^{2}$ with the condition $|\xi|\leq 1$. Then we can write \eqref{KS2} as
\begin{align*}
K_{S,2}(t,x) &= \int_{0}^{t}\int_{|y-x|\leq t-s}\frac{\sigma_{S}(s,y,\xi)K_{g}(s,y,\omega)}{\sqrt{(t-s)^{2}-|y-x|^{2}}} \ dy \ ds 
\end{align*}
where in the above expression, we make use of the convenient shorthand notation $\xi = (y-x)/(t-s)$. The integral in $(s,y)$ is over a (solid) lightcone with height $t$ and base radius $t$. We write $|y-x| = r$ and use the following change of variables (the same as done in \cite{GS1998}):
\begin{align*}
    \psi = \frac{1}{2}(t-s-r),
\end{align*}
which leads to
\begin{align*}
s = t-r-2\psi, \quad  \xi =(r+2\psi)^{-1}(y-x).
\end{align*}
With this choice of variables, we can view the $K_{S,2}$ integral as integrals over conic shells of height $t-2\psi$ in the time variable and base radius $t-2\psi$ in the space variable:
\begin{align}\label{KS2sigmaS}
K_{S,2}(t,x) &= \int_{0}^{t/2}\int_{0}^{t-2\psi}\int_{|y-x|= r}\frac{\sigma_{S}(t-r-2\psi,y,(r+2\psi)^{-1}(y-x))}{\sqrt{\psi}\sqrt{r+\psi}} K_{g}(t-r-2\psi,y,\omega)\ dS_y \ dr\ d\psi.
\end{align}
In particular, writing $K_{S,2}$ in this way will enable us to use the conservation law \eqref{eq:consKg} for $K_g$ over a conic shell.

To estimate $K_{S,2}$, we need to understand the singularity coming from $\sigma_{S}$. For that purpose, using $\angle (v,w)$ to be the counterclockwise measured angle going from $w$ to $v$ with values in $(-\pi,\pi]$, we define
\begin{equation}\label{angles}
\theta = \angle (p,e_1), \quad \vp \eqdef \vp(\xi) = \angle (-\xi,e_1).
\end{equation} 
Note that $\angle(-\xi, p) = \vp-\theta$. This gives
\begin{align*}
    \frac{1}{1+\hat{p}\cdot \xi} &=  \frac{1}{1-\hat{p}\cdot (-\xi)}\\
    &= \frac{1}{1-|\hat{p}||\xi|\cos\angle(-\xi,p)}\\
    &= \frac{1}{1-|\hat{p}||\xi|\cos(\varphi-\theta)}.
\end{align*}
For any $0\leq s \leq t$, the support of $f(s,y,p)$ is contained in $[-P(t),P(t)]\times [-P_{2}(t),P_{2}(t)]$. Thus, we have the bound 
\begin{align}\label{sigmaSbound}
    \sigma_S(s,y,\xi) &\lesssim
    \int_{[-P(t),P(t)]\times [-P_{2}(t),P_{2}(t)]}\frac{1}{p_{0}\left(1-|\hat{p}||\xi|\cos(\varphi-\theta)\right)}\ dp
\end{align}
for any $0\leq s \leq t$.

To control $K_{S,2}$, we must control the denominator in the integral \eqref{sigmaSbound}. It is sufficient to consider $\vp \in [-\pi/2, \pi/2]$ because the complementary region in $\vp$ will yield the same bounds. Hence, we first break down 
\begin{align*}
K_{S,2} &= \tilde{K}_{S,2} + \hat{K}_{S,2}
\end{align*}
where $\tilde{K}_{S,2}$ is the restriction to $\vp\in [-\pi/2,\pi/2]$ and $\hat{K}_{S,2}$ is the complementary integral.

The denominator in \eqref{sigmaSbound} is singular when $|\hat{p}||\xi|\cos(\varphi-\theta) \approx 1$. First, we consider the magnitude of $\hat{p}$. For $p = (p_1,p_2)$, the magnitude of $\hat{p}$ is closest to $1$ near the $e_1$ direction since the domain of integration in \eqref{sigmaSbound} is of the larger size $P(t)$ in that direction. In other words, the denominator is more singular when $|\theta| << 1$ and is not as singular in the complement. Additionally, the denominator is singular when $\cos(\varphi-\theta) \approx 1$, or $|\varphi-\theta| << 1$. Both conditions $|\theta| << 1$  and $|\varphi-\theta| << 1$ are realized in the most singular region where  $|\varphi-\theta| << 1$ and $|\varphi| << 1$.

We define the time dependent quantities:
$$ \alpha \eqdef \frac{P_2(t)}{A(t)}, \quad \beta \eqdef  \frac{P_2(t)}{B(t)}, \quad \gamma \eqdef \frac{P_2(t)}{C(t)}, $$
where for some fixed $\delta > 0$, we impose the condition that $P(t)^{w+\delta} < A(t), B(t), C(t) < P(t) $. The precise expressions for $A(t), B(t), C(t)$ are determined later in the proof. Note that we choose this lower bound on $A(t),B(t)$ and $C(t)$ to ensure that $\alpha,\beta,\gamma \rightarrow 0$ as $t\rightarrow \infty$ since
\begin{align*}
    0 \leq \frac{P_{2}(t)}{P(t)^{w+\delta}} \leq \frac{1}{P(t)^{\delta}}.
\end{align*}
Moreover, this inequality implies that there is some time $T \geq 0$ such that for all $t > T$, we have that $\alpha, \beta, \gamma < \frac{\pi}{4}$. Using these angle definitions, for the term $\tilde{K}_{S,2}$, we decompose the integral in $dS_y$ into the regions $|\varphi| \leq \alpha$ and $|\varphi| > \alpha$ to get $$\tilde{K}_{S,2} =  K_{S,2}^{1} + K_{S,2}^{2}$$
where
\begin{align*}
    K_{S,2}^{1}(t,x)
    &=  \int_{0}^{t/2}\int_{0}^{t-2\psi}\int_{|\varphi| \leq \alpha}\frac{\sigma_{S}(t-r-2\psi,y,(r+2\psi)^{-1}(y-x))}{\sqrt{\psi}\sqrt{r+\psi}} K_{g}(t-r-2\psi,y,\omega)\ dS_y \ dr\ d\psi
\end{align*}
and
\begin{align*}
    K_{S,2}^{2}(t,x)
    &=  \int_{0}^{t/2}\int_{0}^{t-2\psi}\int_{\alpha < |\varphi| \leq \frac{\pi}{2}}\frac{\sigma_{S}(t-r-2\psi,y,(r+2\psi)^{-1}(y-x))}{\sqrt{\psi}\sqrt{r+\psi}} K_{g}(t-r-2\psi,y,\omega)\ dS_y \ dr\ d\psi.
\end{align*}
Subsequently, in $K_{S,2}^{1}(t,x)$, we decompose $\sigma_{S} = \sigma_{S}^{1,1} + \sigma_{S}^{1,2}$ with
\begin{align*}
\sigma_{S}^{1,1}(t,y,\xi) &= \int_{|\theta-\varphi|\leq \beta} \frac{f(t,y,p)}{p_0(1+\hat p\cdot \xi)} \ dp
\end{align*}
and
\begin{align*}
\sigma_{S}^{1,2}(t,y,\xi) &= \int_{|\theta-\varphi| > \beta} \frac{f(t,y,p)}{p_0(1+\hat p\cdot \xi)} \ dp.
\end{align*}
This decomposition gives the respective decomposition $K_{S,2}^{1} = K_{S,2}^{1,1}+K_{S,2}^{1,2}$.

In $K_{S,2}^{2}$, we decompose $\sigma_{S} =  \sigma_{S}^{2,1} + \sigma_{S}^{2,2}$ as
\begin{align*}
\sigma_{S}^{2,1}(t,y,\xi) &= \int_{|\theta|\leq \gamma} \frac{f(t,y,p)}{p_0(1+\hat p\cdot \xi)} \ dp
\end{align*}
and
\begin{align*}
\sigma_{S}^{2,2}(t,y,\xi) &= \int_{|\theta| > \gamma} \frac{f(t,y,p)}{p_0(1+\hat p\cdot \xi)} \ dp,
\end{align*}
which gives the respective decomposition $K_{S,2}^{2} = K_{S,2}^{2,1}+K_{S,2}^{2,2}$.

For each term in the decomposition $K_{S,2}^{i,j}$ with $i,j \in \{1,2\}$, we apply Cauchy-Schwarz integral inequality in the combined $dS_y dr$ integral. Using the conservation law \eqref{eq:consKg} for the term $K_g$, we obtain
\begin{align*}
    K_{S,2}^{1,j}(t,x)
    &\lesssim  \int_{0}^{t/2}\left(\int_{0}^{t-2\psi}\int_{|\varphi| \leq \alpha}\frac{\left(\sigma_{S}^{1,j}(t-r-2\psi,y,(r+2\psi)^{-1}(y-x))\right)^{2}}{\psi(r+\psi)}\ dS_y \ dr\right)^{\frac{1}{2}} d\psi
\end{align*}
and
\begin{align*}
    K_{S,2}^{2,j}(t,x)
    &\lesssim  \int_{0}^{t/2}\left(\int_{0}^{t-2\psi}\int_{\alpha < |\varphi| \leq \frac{\pi}{2}}\frac{\left(\sigma_{S}^{2,j}(t-r-2\psi,y,(r+2\psi)^{-1}(y-x))\right)^{2}}{\psi(r+\psi)}\ dS_y \ dr\right)^{\frac{1}{2}} d\psi.
\end{align*}
We now consider each integral. Using polar coordinates $p = (u\cos\theta, u\sin\theta)$ for $u\in [0,\infty)$ and recalling that
\begin{align*}
 |\xi| &=(r+2\psi)^{-1}|y-x|\\
 &= \frac{r}{r+2\psi},
\end{align*}
we can write the integrand in \eqref{sigmaSbound} as
\begin{align}\label{sigmaSintegrand}
\frac{1}{p_{0}\left(1-|\hat{p}||\xi|\cos(\varphi-\theta)\right)} &= \frac{1}{\sqrt{1+u^2}(1-\frac{u}{\sqrt{1+u^2}} \f{r}{r+2\psi} \cos(\vp-\theta) )}.
\end{align}
This integrand is the same that appears in each $\sigma_{S}^{i,j}$.

\subsubsection{Bounds on $K_{S,2}^{1,1}$}

For $K_{S,2}^{1,1}$, we further split the $\psi$ integral into $\psi\in [0,\varepsilon_{1}]$ and $\psi\in [\varepsilon_{1}, t/2]$ to get the respective decomposition:
\begin{align*}
K_{S,2}^{1,1} &= K_{S,2}^{1,1,1} + K_{S,2}^{1,1,2}.
\end{align*}

It follows from the same argument as \eqref{sigmaSbound} and from \eqref{sigmaSintegrand} that
\begin{align*}
    \sigma_{S}^{1,1}(t-r-2\psi,y,\xi) &\lesssim \int_0^{P(t)} \int_{\vp-\beta}^{\vp+\beta}\frac{u}{\sqrt{1+u^2}(1-\frac{u}{\sqrt{1+u^2}} \f{r}{r+2\psi} \cos(\vp-\theta) )} \ du \ d\theta \\
    & \ls P(t) \int_{\vp-\beta}^{\vp+\beta}\frac{1}{1- \f{r}{r+2\psi} } \ d\theta \\
    & \ls P(t) \frac{\beta}{1- \f{r}{r+2\psi} }\\
    &= P(t) \frac{\beta (r+2\psi)}{2\psi}. 
\end{align*}

Using the above estimate, we can conclude 
\begin{align}
    K_{S,2}^{1,1,2} &\ls P(t) \beta \int_{\varepsilon_1}^{t/2} \frac{1}{\psi^{3/2}} \left( \int_{0}^{t-2\psi}\int_{-\alpha}^\alpha \frac{(r+2\psi)^2}{r+\psi} r \ d\varphi \ dr \right)^{\frac{1}{2}} d\psi \nonumber\\
    & \ls P(t) \alpha^{\frac{1}{2}}  \beta\int_{\varepsilon_1}^{t/2} \frac{1}{{\psi}^{3/2}} \left( \int_{0}^{t-2\psi} \frac{(r+2\psi)^2}{r+\psi} r  \ dr \right)^{\frac{1}{2}} d\psi \nonumber \\
    & \ls  t^{\frac{3}{2}} P(t) \alpha^{\frac{1}{2}} \beta \int_{\varepsilon_1}^{t/2} \frac{1}{{\psi}^{3/2}}  d \psi \nonumber \\
    & \ls   t^{\frac{3}{2}} P(t)  \alpha^{\frac{1}{2}}  \beta\varepsilon_1^{-\frac{1}{2}} \label{KS2112}.
\end{align}

Using the alternative bound $\sigma_{S}^{1,1} \leq \sigma_S \leq P(t)^2$ (see \cite{GS1998}, Lemma 3), we have
\begin{align}
    K_{S,2}^{1,1,1} &\ls  P(t)^{2} \int_{0}^{\varepsilon_1} \frac{1}{\psi^{{\frac{1}{2}}}} \left( \int_{0}^{t-2\psi}\int_{-\alpha}^\alpha \frac{r}{r+\psi}  \ d\varphi \ dr \right)^{\frac{1}{2}} d\psi \nonumber\\
    &\lesssim t^{\frac{1}{2}} P(t)^{2}\alpha^{\frac{1}{2}} \int_{0}^{\varepsilon_1}\frac{1}{\psi^{{\frac{1}{2}}}}  d\psi \nonumber\\
    &\lesssim  t^{\frac{1}{2}} P(t)^{2}\alpha^{\frac{1}{2}} \varepsilon_{1}^{\frac{1}{2}} \label{KS2111}.
\end{align}
We optimize the sum of the bounds \eqref{KS2112} and \eqref{KS2111} in the parameter $\varepsilon_1$:
\begin{equation*}
     t^{\frac{3}{2}} P(t)  \alpha^{\frac{1}{2}}  \beta\varepsilon_1^{-\frac{1}{2}} = t^{\frac{1}{2}} P(t)^{2}\alpha^{\frac{1}{2}} \varepsilon_{1}^{\frac{1}{2}}
\end{equation*}
which yields
\begin{equation*}
    \varepsilon_1 = t P(t)^{-1} \beta,
\end{equation*}
and hence
\begin{align}\label{KS211est}
 K_{S,2}^{1,1} &\ls t P(t)^{\frac{3}{2}}\alpha^{\frac{1}{2}}\beta^{\frac{1}{2}} = t P_{2}(t) P(t)^{\frac{3}{2}} A(t)^{-\frac{1}{2}}B(t)^{-\frac{1}{2}}.
\end{align}

\subsubsection{Bounds on $K_{S,2}^{1,2}$}
For $K_{S,2}^{1,2}$, we further split the $\psi$ integral into $\psi\in [0,\varepsilon_{2}]$ and $\psi\in [\varepsilon_{2}, t/2]$ to get the respective decomposition:
\begin{align*}
K_{S,2}^{1,2} &= K_{S,2}^{1,2,1} + K_{S,2}^{1,2,2}.
\end{align*}
In considering $K_{S,2}^{1,2}$ bounds, we will impose 
that
\begin{align}\label{rest1}
  \frac{\beta}{2} \geq \alpha, 
\end{align}
which also implies the relation $A \geq 2B$. Then,
\begin{align*}
|\theta| = |\theta-\varphi+\varphi|\geq |\theta-\varphi| - |\vp|\geq \beta - \alpha \ge \frac{\beta}{2}.
\end{align*}
First, we consider $p \in [-P(t), P(t)]\times [0, P_{2}(t)]$, so it holds that
\begin{align}\label{eq:boundp}
|p| = \frac{p_{2}}{\sin\theta} \leq  \frac{P_{2}(t)}{\sin\theta}.
\end{align}
We estimate $\sigma_S^{1,2}$ as follows. In the region where $\beta/2 \leq  \theta < \pi-\beta/2$, we have that
\begin{align}
\int\displaylimits_{\substack{\theta\in [\beta/2, \pi-\beta/2)\\
p \in [-P(t), P(t)]\times [0, P_{2}(t)]}} \frac{1}{p_{0}(1+\hat{p}\cdot\xi)} dp &= \int_{\beta/2}^{\pi-\beta/2} \int_0^{\frac{P_2(t)}{\sin\theta}}\frac{u}{\sqrt{1+u^2}(1-\frac{u}{\sqrt{1+u^2}} \f{r}{r+2\psi} \cos(\vp-\theta) )} \ du \ d\theta     \nonumber    \\
&\leq  \int_{\beta/2}^{\pi-\beta/2}  \int_0^{\frac{P_2(t)}{\sin\theta}} \frac{1}{1- \f{r}{r+2\psi} } \ du \ d\theta \nonumber \\
         &=  P_2(t) \frac{r+2\psi}{2\psi}  \int_{\beta/2}^{\pi-\beta/2} \frac{1}{\sin\theta} d \theta \nonumber \\
         &= P_2(t)  \frac{r+2\psi}{\psi}\log \left( \cot \left(\f{\beta}{4} \right) \right) \label{eq:sigma_12a}.
\end{align}
At the same time, we can use our choice of $\beta$ to say that
\begin{align*}
|p| = \frac{p_{2}}{\sin\theta} \leq \frac{P_{2}(t)}{\sin(\beta/2)} \lesssim B(t)
\end{align*}
for $\theta\in [\beta/2, \pi-\beta/2)$. Notice that to get a sharper bound than $\sigma_S^{1,2} \leq \sigma_{S} \lesssim P(t)^{2}$, it makes sense that we should pick $B(t) < P(t)$. Then, following the same argument as in Lemma 3 in \cite{GS1998}, we deduce
\begin{align}
       \int\displaylimits_{\substack{\theta\in [\beta/2, \pi-\beta/2)\\
p \in [-P(t), P(t)]\times [0, P_{2}(t)]}} \frac{1}{p_{0}(1+\hat{p}\cdot\xi)} dp 
        &\lesssim \int_{\beta/2}^{\pi-\beta/2} \int_0^{B(t)}\frac{u}{\sqrt{1+u^2}(1-\frac{u}{\sqrt{1+u^2}} \f{r}{r+2\psi} \cos(\vp-\theta) )} \ du \ d\theta \nonumber \\
        &\leq \int_0^{B(t)}\int_{0}^{\pi}\frac{u}{\sqrt{1+u^2}(1- \frac{u}{\sqrt{1+u^2}} \cos(\vp-\theta) )} \ d\theta \ du \nonumber \\
        &\lesssim \int_0^{B(t)}\frac{u}{\sqrt{1+u^2}}\frac{1}{\sqrt{1-\frac{u^{2}}{1+u^{2}}}}  \ du \nonumber\\
         &\leq \int_0^{B(t)} u  \ du \nonumber\\
        &\lesssim B(t)^{2} \label{eq:sigma_12b}.
\end{align}

In the region $\pi-\beta/2 \leq  \theta \leq  \pi$, we have that 
\begin{align*}
    |\theta -\varphi| \geq | |\theta| - |\varphi|| \geq \pi - \beta/2 -\alpha \geq \frac{\pi}{2}
\end{align*}
and 
\begin{align*}
    |\theta -\varphi| \leq |\theta| + |\varphi| \leq \pi + \alpha \leq \frac{3\pi}{2},
 \end{align*}
and hence
\begin{align*}
    1+\hat{p}\cdot \xi &= 1 - |\hat{p}||\xi|\cos(\varphi-\theta) \geq 1,
\end{align*}
so we can estimate
\begin{align}
\int\displaylimits_{\substack{\theta\in [\pi-\beta/2,\pi]\\
p \in [-P(t), P(t)]\times [0, P_{2}(t)]}} \frac{1}{p_{0}(1+\hat{p}\cdot\xi)} dp &\leq \int_{0}^{P_{2}(t)}\int_{-P(t)}^{P(t)} \frac{1}{p_{0}} \  dp_1 \ dp_2 \nonumber\\
&\lesssim P_{2}(t)\log(P(t)). \label{eq:sigma_12c}
\end{align}
 Using the bounds \eqref{eq:sigma_12a}, \eqref{eq:sigma_12b} and \eqref{eq:sigma_12c} and using the same bounds that can be obtained by symmetry for the region $\theta \in [-\pi, -\beta/2]$ i.e., when $p \in [-P(t),P(t)] \times [-P_2(t),0]$, we obtain that
 \begin{align*}
    \left(\sigma_{S}^{1,2}(t,y,\xi)\right)^{2} 
&\lesssim P_2(t) B(t)^{2} \frac{r+2\psi}{\psi} \log \left( \cot \left( \f{\beta}{4} \right) \right) + P_{2}(t)^{2}\log(P(t))^{2}
 \end{align*}
 and hence, for $t$ sufficiently large,
\begin{align}
    K_{S,2}^{1,2,2} &\lesssim  \int_{\varepsilon_2}^{t/2} \frac{1}{\sqrt{\psi}} \left( \int_{0}^{t-2\psi}\int_{|\vp|\leq\alpha}\left( P_2(t) B(t)^{2}\frac{r+2\psi}{\psi} \log \left( \cot \left( \f{\beta}{4} \right) \right) +P_{2}(t)^{2}\log(P(t))^{2}\right)\f{r}{r+\psi} \ d\varphi \ dr \right)^{\frac{1}{2}} d\psi \nonumber\\
    &\lesssim  P_2(t)^{\frac{1}{2}}B(t)\log \left( \cot \left( \f{\beta}{4} \right) \right)^{\frac{1}{2}}\int_{\varepsilon_2}^{t/2} \frac{1}{\psi} \left( \int_{0}^{t-2\psi}\int_{|\vp|\leq \alpha} (r+2\psi) \f{r}{r+\psi} \ d\varphi \ dr \right)^{\frac{1}{2}} d\psi \nonumber\\
    &\qquad + P_2(t)\log(P(t))\int_{\varepsilon_2}^{t/2} \frac{1}{\sqrt{\psi}} \left( \int_{0}^{t-2\psi}\int_{|\vp|\leq \alpha}  \f{r}{r+\psi} \ d\varphi \ dr \right)^{\frac{1}{2}} d\psi\nonumber\\
    &\lesssim t  P_2(t)^{\frac{1}{2}} \alpha^{\frac{1}{2}} B(t) \log \left( \cot \left( \f{\beta}{4} \right) \right)^{\frac{1}{2}} \left(-\log(\varepsilon_2)+\log(t)\right) + t P_{2}(t)\alpha^{\frac{1}{2}} \log(P(t)).\label{KS2122bound}
\end{align}
We have that for $t$ sufficiently large,
\begin{align*}
\log \left( \cot \left( \f{\beta}{4} \right) \right) &= \log \left( \cos \left( \f{\beta}{4} \right) \right) - \log \left( \sin \left( \f{\beta}{4} \right) \right) \\
&\lesssim \left| \log \left( \sin \left( \f{\beta}{4} \right) \right) \right|\\
&\lesssim \left| \log \left( \beta\right) \right|\\
&\leq \log(B(t)) - \log(P_{2}(t))\\
&\lesssim \log(P(t)),
\end{align*}
and also that $P_{2}(t)^{\frac{1}{2}} \leq P_{2}(t) \leq B(t)$. We will choose $\varepsilon_2 = t^{-1}  \left(B(t)^{2} + P_{2}(t)\log(P(t))\right)^{-2}$ so that $-\log(\varepsilon_{2}) \lesssim \log(P(t))$. So our bound becomes
\begin{align}\label{KS2122}
 K_{S,2}^{1,2,2} &\lesssim t P_2(t)\log (P(t))^{\frac{3}{2}} A(t)^{-\frac{1}{2}} B(t) .
\end{align}

Furthermore, using just \eqref{eq:sigma_12b} and \eqref{eq:sigma_12c}, we find that 
\begin{align}
    K_{S,2}^{1,2,1} &\lesssim  \left(B(t)^{2} + P_{2}(t)\log(P(t))\right)\int_{0}^{\varepsilon_2} \frac{1}{\sqrt{\psi}} \left(\int_{0}^{t-2\psi}\int_{|\vp|\leq \alpha} \ d\vp \ dr \right)^{\frac{1}{2}} d\psi \nonumber\\
    &\lesssim t^{\frac{1}{2}} \alpha^{\frac{1}{2}}  \left(B(t)^{2} + P_{2}(t)\log(P(t))\right)\varepsilon_2^{\frac{1}{2}}. \label{eq:KS2121boundestimate}
\end{align}

In this case, because $\varepsilon_2 = t^{-1}  \left(B(t)^{2} + P_{2}(t)\log(P(t))\right)^{-2}$, we get insignificant contribution from $K_{S,2}^{1,2,1}$ since $K_{S,2}^{1,2,1} \lesssim \alpha^{\frac{1}{2}}$. Hence, we obtain
\begin{align}\label{KS212est}
K_{S,2}^{1,2} &\lesssim t P_2(t) \log (P(t))^{\frac{3}{2}} A(t)^{-\frac{1}{2}} B(t)  .
\end{align}

\subsubsection{Bounds on $K_{S,2}^{2,1}$}
For $K_{S,2}^{2,1}$, we further split the $\psi$ integral into $\psi\in [0,\varepsilon_{3}]$ and $\psi\in [\varepsilon_{3}, t/2]$ to get the respective decomposition:
\begin{align*}
K_{S,2}^{2,1} &= K_{S,2}^{2,1,1} + K_{S,2}^{2,1,2}.
\end{align*}
We impose a restriction on our choice of $\gamma$ to satisfy 
\begin{align} \label{rest2}
 \gamma \leq \frac{\alpha}{2}
\end{align}
similarly as in \eqref{rest1}. This restriction implies that $2 A \leq C$ and will be used later in the estimate for $K_{S,2}^{2,1}$. We bound
\begin{align}
        \sigma_{S}^{2,1}(t-r-2\psi,y,\xi) &\lesssim \int_0^{P(t)} \int_{-\gamma}^{\gamma}\frac{u}{\sqrt{1+u^2}(1-\frac{u}{\sqrt{1+u^2}} \f{r}{r+2\psi} \cos(\vp-\theta) )} \ du \ d\theta \nonumber \\
        & \ls P(t) \int_{\vp-\gamma}^{\vp+\gamma}\frac{1}{1- \f{r}{r+2\psi}\cos\lambda }  \ d\lambda \nonumber\\
        & \ls P(t) \frac{1}{\sqrt{1-\frac{r^2}{(r+2\psi)^2}}} \tan^{-1} \left( \frac{ \left(1+\frac{r}{r+2\psi} \right) \tan(\lambda/2)}{ \sqrt{1-\frac{r^2}{(r+2\psi)^2}} } \right) \Bigg|_{\vp-\gamma}^{\vp+\gamma} \nonumber \\
        & = P(t) \frac{r+2\psi}{2\sqrt{\psi(r+\psi)}} \tan^{-1} \left( \frac{\sqrt{r+\psi}}{\sqrt{\psi}} \tan(\lambda/2) \right)\Bigg|_{\vp-\gamma}^{\vp+\gamma} \nonumber\\
        & \ls P(t) \gamma \frac{r+2 \psi}{\psi} \label{sigma21est} ,
\end{align}
where we have used that the slope of $\tan^{-1} x$ is bounded by $1$ and that the derivative of $\tan x$ gives
\begin{align}\label{tangentest}
\tan\left(\frac{\vp+\gamma}{2}\right)- \tan\left(\frac{\vp-\gamma}{2}\right) &\leq  \sup_{\lambda\in [\vp-\gamma,\vp+\gamma]}\frac{1}{\cos^{2}\left(\frac{\lambda}{2}\right)} \gamma \lesssim \gamma,
\end{align}
since $|\lambda/2|\leq \frac{3\pi}{8}$ due to
\begin{align*}
\left|\frac{\vp \pm \gamma}{2}\right| \leq \frac{1}{2}\left(|\vp| + |\gamma|\right) \leq \frac{1}{2}\left(\frac{\pi}{2} + \frac{\pi}{4}\right) = \frac{3\pi}{8} < \frac{\pi}{2}.
\end{align*}

Using \eqref{tangentest} gives another useful estimate on the term involving $\tan^{-1}$, shown below:
\begin{align*}
   \tan^{-1} \left( \frac{\sqrt{r+\psi}}{\sqrt{\psi}} \tan \left(\frac{\lambda}{2}\right) \right) \Bigg|_{\vp-\gamma}^{\vp+\gamma} &\leq  \frac{\sqrt{r+\psi}}{\sqrt{\psi}} \left( \tan \left(\frac{\varphi+\gamma}{2} \right) - \tan \left(\frac{\varphi-\gamma}{2} \right)\right)\sup_{\lambda\in(\varphi-\gamma,\varphi+\gamma)}\left(\frac{1}{1+\frac{r+\psi}{\psi}\tan^{2}(\lambda/2)}\right)\\
&\lesssim \gamma \frac{\sqrt{r+\psi}}{\sqrt{\psi}}\sup_{\lambda\in(\varphi-\gamma,\varphi+\gamma)}\left(\frac{1}{1+\frac{r+\psi}{\psi}\tan^{2}(\lambda/2)}\right).
\end{align*}
Combining the two estimates gives
\begin{align}\label{sigma21precise}
 \sigma_{S}^{2,1}(t-r-2\psi,y,\xi) &\lesssim P(t)\gamma \frac{r+2\psi}{\psi} \sup_{\lambda\in(\varphi-\gamma,\varphi+\gamma)}\left(\frac{1}{1+\frac{r+\psi}{\psi}\tan^{2}(\lambda/2)}\right).
\end{align}
Next, since
\begin{align*}
 \alpha \leq |\varphi| \leq \frac{\pi}{2},
\end{align*}
we then know from the assumption $\gamma \leq \frac{\alpha}{2}$ that $$|\varphi-\gamma| \geq  |\vp| - \gamma \geq \alpha - \gamma \geq \frac{\alpha}{2}, \qquad |\varphi+\gamma| \leq  |\vp| + \gamma \leq \frac{\pi}{2} + \gamma \leq \frac{\pi}{2} + \gamma.$$
Hence, we are assured that $[\vp-\gamma, \vp +\gamma]$ is either contained in $[\alpha/2,\pi/2+\gamma]$ or $[-\pi/2-\gamma, -\alpha/2]$, so the supremum in \eqref{sigma21precise} occurs at the endpoint of the interval $[\vp-\gamma, \vp +\gamma]$ closer to $0$. Thus, we can get the following estimate:
\begin{align*}
\int_{\alpha}^{\pi/2} d\varphi \sup_{\lambda\in(\varphi-\gamma,\varphi+\gamma)}\left(\frac{1}{1+\frac{r+\psi}{\psi}\tan^{2}(\lambda/2)}\right) 
&= \int_{\alpha}^{\pi/2} d\varphi \left(\frac{1}{1+\frac{r+\psi}{\psi}\tan^{2}(
\frac{\varphi-\gamma}{2})}\right)\\
&= 2\frac{\sqrt{\frac{r+\psi}{\psi}}\tan^{-1}\left(\sqrt{\frac{r+\psi}{\psi}}\tan(\varphi)\right)-\varphi}{\frac{r+\psi}{\psi}-1}\Bigg|_{\frac{\alpha-\gamma}{2}}^{\frac{\frac{\pi}{2}-\gamma}{2}}\\
&\lesssim \frac{\sqrt{\frac{r+\psi}{\psi}}}{\frac{r+\psi}{\psi}-1} + \frac{\psi}{r} \\
&= \frac{\sqrt{\psi}\sqrt{r+\psi}}{r} + \frac{\psi}{r}.
\end{align*}
Similarly,
\begin{align*}
\int_{-\pi/2}^{-\alpha} d\varphi \sup_{\lambda\in(\varphi-\gamma,\varphi+\gamma)}\left(\frac{1}{1+\frac{r+\psi}{\psi}\tan^{2}(\lambda/2)}\right) 
&= \int_{-\pi/2}^{-\alpha} d\varphi \left(\frac{1}{1+\frac{r+\psi}{\psi}\tan^{2}(
\frac{\varphi+\gamma}{2})}\right)\\
&\lesssim \frac{\sqrt{\psi}\sqrt{r+\psi}}{r} + \frac{\psi}{r}.
\end{align*}
Hence, 
\begin{align}\label{sigma21integral}
\int_{\alpha \le |\vp|\le \frac{\pi}{2}} d\varphi \sup_{\lambda\in(\varphi-\gamma,\varphi+\gamma)}\left(\frac{1}{1+\frac{r+\psi}{\psi}\tan^{2}(\lambda/2)}\right) 
&\lesssim \frac{\sqrt{\psi}\sqrt{r+\psi}}{r} + \frac{\psi}{r} \lesssim  \frac{\sqrt{\psi}\sqrt{r+\psi}}{r}.
\end{align}

From here, we can bound $K_{S,2}^{2,1,2}$. We use the bounds \eqref{sigma21est}, \eqref{sigma21precise} and \eqref{sigma21integral} to obtain
\begin{align*}
    K_{S,2}^{2,1,2} &\ls  \int_{\varepsilon_3}^{t/2}  \left( \int_{0}^{t-2\psi}\int_{\alpha < |\vp| \leq \frac{\pi}{2}} P(t)^{2}\gamma^{2} \frac{(r+2\psi)^{2}}{\psi^{3} (r+\psi)}\sup_{\lambda\in(\varphi-\gamma,\varphi+\gamma)}\left(\frac{1}{1+\frac{r+\psi}{\psi}\tan^{2}(\lambda/2)}\right) r \ d\varphi \ dr \right)^{\frac{1}{2}} d\psi \nonumber \\
&\lesssim P(t)\gamma\int_{\varepsilon_3}^{t/2} \frac{1}{\psi^{3/2}} \left( \int_{0}^{t-2\psi}\frac{r(r+2\psi)^{2}}{(r+\psi)}\int_{\alpha < |\vp| \leq \frac{\pi}{2}} \sup_{\lambda\in(\varphi-\gamma,\varphi+\gamma)}\left(\frac{1}{1+\frac{r+\psi}{\psi}\tan^{2}(\lambda/2)}\right)   \ d\varphi \ dr \right)^{\frac{1}{2}} d\psi \nonumber \\
&\lesssim P(t)\gamma \int_{\varepsilon_3}^{t/2}  \frac{1}{\psi^{3/2}}  \left( \int_{0}^{t-2\psi}\frac{r(r+2\psi)^{2}}{(r+\psi)}\frac{\sqrt{\psi}\sqrt{r+\psi}}{r} \ dr \right)^{\frac{1}{2}} d\psi ,
\end{align*}
and then we bound the resulting integral similarly as in the previous cases,
\begin{align}
    K_{S,2}^{2,1,2} 
&\lesssim P(t)\gamma \int_{\varepsilon_3}^{t/2}  \frac{1}{\psi^{3/2}}  \left( \int_{0}^{t-2\psi} \frac{(r+2\psi)^{2} \sqrt{\psi}}{\sqrt{r+\psi}}\ dr \right)^{\frac{1}{2}} d\psi \nonumber\\
&\leq P(t)\gamma \int_{\varepsilon_3}^{t/2}  \frac{1}{\psi^{5/4}}  \left( \int_{0}^{t-2\psi} \frac{(r+2\psi)^{2} }{\sqrt{\frac{r}{2}+\psi}}\ dr \right)^{\frac{1}{2}} d\psi \nonumber \\
&\lesssim P(t)\gamma \int_{\varepsilon_3}^{t/2}  \frac{1}{\psi^{5/4}}  \left( \int_{0}^{t-2\psi} (r+2\psi)^{3/2} \ dr \right)^{\frac{1}{2}} d\psi \nonumber \\
&\lesssim t^{\frac{5}{4}} P(t)\gamma \varepsilon_3^{-\frac{1}{4}}. \label{KS2212}
\end{align}

Next, we will bound $K_{S,2}^{2,1,1}$ to obtain an optimization in $\varepsilon_3$. Note that we will use the rougher estimate $\sigma_{S}^{2,1}\lesssim P(t)^{2}$ to reduce the singularity in the $\psi$ variable. Hence, we use \eqref{sigma21precise} and \eqref{sigma21integral} to obtain that
\begin{align*}
\sigma_{S}^{2,1}(t-r-2\psi, y, \xi)^{2} &\lesssim P(t)^{3}\gamma \frac{r+2\psi}{\psi} \sup_{\lambda\in(\varphi-\gamma,\varphi+\gamma)}\left(\frac{1}{1+\frac{r+\psi}{\psi}\tan^{2}(\lambda/2)}\right)
.\end{align*}
Thus,

\begin{align}
K_{S,2}^{2,1,1}  &\ls  \int_{0}^{\varepsilon_3} \frac{1}{\sqrt{\psi}} \left( \int_{0}^{t-2\psi}\int_{\alpha < |\vp|\leq \frac{\pi}{2}} P(t)^{3}\gamma \frac{r+2\psi}{\psi} \sup_{\lambda\in(\varphi-\gamma,\varphi+\gamma)}\left(\frac{1}{1+\frac{r+\psi}{\psi}\tan^{2}(\lambda/2)}\right) \frac{r}{r+\psi} \ d\varphi \ dr \right)^{\frac{1}{2}} d\psi \nonumber \\
&\ls  P(t)^{\frac{3}{2}}\gamma^{\frac{1}{2}}\int_{0}^{\varepsilon_3} \frac{1}{\psi} \left( \int_{0}^{t-2\psi}\frac{r(r+2\psi)}{r+\psi}\int_{\alpha < |\vp|\leq \frac{\pi}{2}}  \sup_{\lambda\in(\varphi-\gamma,\varphi+\gamma)}\left(\frac{1}{1+\frac{r+\psi}{\psi}\tan^{2}(\lambda/2)}\right)\ d\varphi \ dr \right)^{\frac{1}{2}} d\psi  \nonumber\\
&\ls  P(t)^{\frac{3}{2}}\gamma^{\frac{1}{2}}\int_{0}^{\varepsilon_3} \frac{1}{\psi} \left( \int_{0}^{t-2\psi}\frac{(r+2\psi)}{r+\psi} \sqrt{\psi}\sqrt{r+\psi}  \ dr \right)^{\frac{1}{2}} d\psi \nonumber\\
&\ls  P(t)^{\frac{3}{2}}\gamma^{\frac{1}{2}}\int_{0}^{\varepsilon_3} \frac{1}{\psi^{\frac{3}{4}}} \left( \int_{0}^{t-2\psi}(r+2\psi)^{\frac{1}{2}}  \ dr \right)^{\frac{1}{2}} d\psi \nonumber\\
&\ls t^{\frac{3}{4}} P(t)^{\frac{3}{2}}\gamma^{\frac{1}{2}} \varepsilon_{3}^{\frac{1}{4}}. \label{KS2211}
\end{align}
To optimize the sum of \eqref{KS2212} and \eqref{KS2211} in the parameter $\varepsilon_3$, we set
\begin{equation*}
 t^{\frac{5}{4}}  P(t)\gamma \varepsilon_{3}^{-\frac{1}{4}} =  t^{\frac{3}{4}} P(t)^{\frac{3}{2}}\gamma^{\frac{1}{2}} \varepsilon_{3}^{\frac{1}{4}}
\end{equation*}
and so
\begin{equation*}
  \varepsilon_{3}^{\frac{1}{2}} =  t^{\frac{1}{2}}P(t)^{-\frac{1}{2}}\gamma^{\frac{1}{2}}.
\end{equation*}
Therefore,
\begin{align}\label{KS221est}
K_{S,2}^{2,1} &\lesssim t P(t)^{\frac{5}{4}}\gamma^{\frac{3}{4}} \lesssim t P_{2}(t)^{\frac{3}{4}}P(t)^{\frac{5}{4}} C(t)^{-\frac{3}{4}}.
\end{align}

\subsubsection{Bounds on $K_{S,2}^{2,2}$}

In $K_{S,2}^{2,2}$, we decompose the $\psi$ integral into $[0,\varepsilon_4]$ and $[\varepsilon_4, t/2]$
\begin{align*}
  K_{S,2}^{2,2}=  K_{S,2}^{2,2,1}+  K_{S,2}^{2,2,2}.
\end{align*}

We recall that the support of $\theta$ for this term is such that 
$$ |\theta| > \gamma.$$
Therefore, similar computations as for $K_{S,2}^{1,2}$ hold. Firstly, similarly as in \eqref{eq:boundp} we have that for $\theta \in (\gamma, \pi-\gamma)$ and $p \in [-P(t), P(t)]\times [0, P_{2}(t)]$,
\begin{align*}
|p| = \frac{p_{2}}{\sin\theta} \leq  \frac{P_{2}(t)}{\sin\theta} \leq \frac{P_{2}(t)}{\sin(\gamma)} \lesssim C(t),
\end{align*}
given our choice of $\gamma $.
In the same range for $\theta$, using the same estimate as in \eqref{eq:sigma_12a}, we obtain
\begin{align*}
        \int\displaylimits_{\substack{\theta\in [\gamma, \pi-\gamma)\\
p \in [-P(t), P(t)]\times [0, P_{2}(t)]}} \frac{1}{p_{0}(1+\hat{p}\cdot\xi)} dp &
          \ls P_2(t)  \frac{r+2\psi}{\psi} \log \left( \cot \left( \f{\gamma}{2} \right) \right) 
\end{align*}
and by \eqref{eq:sigma_12b},
\begin{align*}
\int\displaylimits_{\substack{\theta\in [\gamma, \pi-\gamma)\\
p \in [-P(t), P(t)]\times [0, P_{2}(t)]}} \frac{1}{p_{0}(1+\hat{p}\cdot\xi)} dp  \lesssim C(t)^2.
\end{align*}

When $\theta \in [\pi-\gamma, \pi]$, 
$$ |\theta-\varphi| \ge \pi - \gamma - \frac{\pi}{2} \ge \frac{\pi}{2}-\gamma \ge \frac{\pi}{3}$$
$$ |\theta-\varphi| \leq \frac{3\pi}{2}. $$
Similarly as in \eqref{eq:sigma_12c},
\begin{align}
\int\displaylimits_{\substack{\theta\in [\pi-\gamma,\pi]  \\
p \in [-P(t), P(t)]\times [0, P_{2}(t)]}} \frac{1}{p_{0}(1+\hat{p}\cdot\xi)} dp & \lesssim P_2(t) \log (P(t)).
\end{align}
The same bounds can be shown for  $p \in [-P(t), P(t)]\times [-P_{2}(t),0]$ and $\theta \in (-\gamma, -\pi)$. Therefore, 
\begin{align*}
    \left(\sigma_{S}^{2,2}(t,y,\xi)\right)^{2} 
&\lesssim P_2(t) C(t)^{2} \frac{r+2\psi}{\psi} \log \left( \cot \left( \f{\gamma}{2} \right) \right) + P_{2}(t)^{2}\log(P(t))^{2}.
 \end{align*}
 
Using these bounds, we obtain that
\begin{align}
    K_{S,2}^{2,2,2} &\lesssim  \int_{\varepsilon_4}^{t/2} \frac{1}{\sqrt{\psi}} \left( \int_{0}^{t-2\psi}\int_{\alpha <|\vp|\leq \frac{\pi}{2}} \left(P_2(t) C(t)^2\frac{r+2\psi}{\psi} \log \left( \cot \left( \f{\gamma}{2} \right) \right)+P_{2}(t)^{2}\log(P(t))^{2}\right) \f{r}{r+\psi} \ d\varphi \ dr \right)^{\frac{1}{2}} d\psi \nonumber\\
    &\lesssim P_2(t)^{\frac{1}{2}}C(t)\log \left( \cot \left( \f{\gamma}{2} \right) \right)^{\frac{1}{2}}\int_{\varepsilon_4}^{t/2} \frac{1}{\psi} \left( \int_{0}^{t-2\psi}\int_{\alpha <|\vp|\leq \frac{\pi}{2}} (r+2\psi) \f{r}{r+\psi} \ d\varphi \ dr \right)^{\frac{1}{2}} d\psi \nonumber\\
    &\qquad + P_{2}(t)\log(P(t)) \int_{\varepsilon_4}^{t/2} \frac{1}{\sqrt{\psi}} \left( \int_{0}^{t-2\psi}\int_{\alpha <|\vp|\leq \frac{\pi}{2}}  \f{r}{r+\psi} \ d\varphi \ dr \right)^{\frac{1}{2}} d\psi \nonumber\\
    &\lesssim t P_2(t)^{\frac{1}{2}}  C(t)\log \left( \cot \left( \f{\gamma}{2} \right) \right)^{\frac{1}{2}} \left(-\log(\varepsilon_4) + \log(t) \right) + t P_2(t)\log(P(t))  \nonumber\\
    &\lesssim  t P_2(t)^{\frac{1}{2}}  \log (P(t))^{\frac{3}{2}}C(t),\label{KS2222bound}
\end{align}
where the last line is obtained by an argument similar to the one used to get \eqref{KS2122} using $\varepsilon_4 = t^{-1}(C(t)^2+ P_{2}(t)\log(P(t)))^{-2}$ . Additionally, as in \eqref{eq:KS2121boundestimate}, 
\begin{align*}
    K_{S,2}^{2,2,1} &\lesssim \int_{0}^{\varepsilon_4} \frac{1}{\sqrt{\psi}} \left(\int_{0}^{t-2\psi}\int_{\alpha <|\vp|\leq \frac{\pi}{2}}\frac{\sigma_{S}^{2,2}(t-r-2\psi,y,(r+2\psi)^{-1}(y-x))^{2}}{{r+\psi}} dS_y \ dr \right)^{\frac{1}{2}} d\psi\\
    &\lesssim t^{\frac{1}{2}}(C(t)^2+ P_{2}(t)\log(P(t)))\varepsilon_4^{\frac{1}{2}}.
\end{align*}

Notice that we take $\varepsilon_4 = t^{-1}(C(t)^2+ P_{2}(t)\log(P(t)))^{-2}$ to get insignificant contribution from $K_{S,2}^{2,2,1}$ as we did in the argument following \eqref{eq:KS2121boundestimate}. Hence,
\begin{align}\label{KS222est}
    K_{S,2}^{2,2} &\lesssim t P_2(t)^{\frac{1}{2}} \log(P(t))^{\frac{3}{2}} C(t) .
\end{align}

\subsubsection{Optimizing the Bounds for $K_{S,2}$}

At this point, we will use the bounds for $K_{S,2}^{i,j}$ derived in the previous sections in order to get a final estimate for $K_{S,2}$. The estimates obtained for each of the terms depend on the parameters $\alpha, \beta, \gamma$, or equivalently $A= A(t)$, $B = B(t)$ and $C = C(t)$. From \eqref{rest1} and \eqref{rest2}, we know that $4B \leq 2A \leq C$ must hold. We optimize the bound on $K_{S,2}$ with respect to this condition on the parameters.

First, we look at the sum \eqref{KS211est} and \eqref{KS212est}, which give estimates for $K_{S,2}^{1,1}$ and $K_{S,2}^{1,2}$. The sum reads:
\begin{align*}
K_{S,2}^{1,1}+K_{S,2}^{1,2}&\lesssim t P_{2}(t) P(t)^{\frac{3}{2}} A^{-\frac{1}{2}}B^{-\frac{1}{2}} +  t   P_2(t)\log (P(t))^{\frac{3}{2}}  A^{-\frac{1}{2}} B\\
&= t P_{2}(t) A^{-\frac{1}{2}}\left( P(t)^{\frac{3}{2}} B^{-\frac{1}{2}}+ \log (P(t))^{\frac{3}{2}} B\right).
\end{align*}
Thus, to minimize the bound, we choose the minimal $A$, which is $2A=C$. Thus, together with the estimates \eqref{KS221est} and \eqref{KS222est}, the sum of all four terms is
\begin{align*}
K_{S,2} &\lesssim  t P_{2}(t)^{\frac{3}{4}}P(t)^{\frac{5}{4}} C^{-\frac{3}{4}} + t  P_2(t)^{\frac{1}{2}} \log(P(t))^{\frac{3}{2}}C + t P_{2}(t) C^{-\frac{1}{2}}\left( P(t)^{\frac{3}{2}} B^{-\frac{1}{2}}+ \log (P(t))^{\frac{3}{2}} B\right)\\
&= t C^{-\frac{1}{2}}\left[\left( P_{2}(t)^{\frac{3}{4}}P(t)^{\frac{5}{4}} C^{-\frac{1}{4}} + P_2(t)^{\frac{1}{2}} \log(P(t))^{\frac{3}{2}}C^{\frac{3}{2}}\right) + P_{2}(t)\left(P(t)^{\frac{3}{2}} B^{-\frac{1}{2}}+ \log (P(t))^{\frac{3}{2}}B\right)\right]\\
&\eqdef h(B,C).
\end{align*}
We must now minimize $h(B,C)$ under the condition $0 < 4B\leq C$. Noting that $h(B,C)$ tends to $+\infty$ as $B$ or $C$ tend to $0$ or $+\infty$, any local minimum for $h(B,C)$ in the quadrant $(0,\infty)\times (0,\infty)$ is achieved at a critical point. For the $B$ coordinate of a critical point, we have
\begin{align*}
    0 &= \frac{d}{dB} \left(P(t)^{\frac{3}{2}} B^{-\frac{1}{2}}  +  \log(P(t))^{\frac{3}{2}}B\right)\\
    &= -\frac{1}{2} P(t)^{\frac{3}{2}}B^{-\frac{3}{2}} + \log(P(t))^{\frac{3}{2}}.
\end{align*}
and so,
\begin{align*}
    B &=  2^{-\frac{2}{3}}\log(P(t))^{-1} P(t).
\end{align*}
However, we only want to consider the subregion $0 < 4B\leq C$. Therefore any critical point that is in this subregion must satisfy  $4\log(P(t))^{-1} P(t)\leq C$ and
\begin{align}\label{lowerboundhBC}
h(B,C) \geq t P_2(t)^{\frac{1}{2}} \log(P(t))^{\frac{3}{2}}C \geq 4 t P_2(t)^{\frac{1}{2}} \log(P(t))^{\frac{1}{2}} P(t).
\end{align}
 We now check the value of $h(B,C)$ on the boundary $4B= C$. Setting $B= \frac{1}{4}C$ in the sum, we have
\begin{align}\label{totalboundKS2}
h(B,C) &\approx t P_{2}(t)^{\frac{3}{4}}P(t)^{\frac{5}{4}} C^{-\frac{3}{4}} +  t P_2(t)^{\frac{1}{2}}   \log(P(t))^{\frac{3}{2}} C + t P_{2}(t) P(t)^{\frac{3}{2}} C^{-1}+ t P_{2}(t) \log (P(t))^{\frac{3}{2}}C^{\frac{1}{2}} .
\end{align}
Notice that the second and third terms are larger than the first and fourth terms for sufficiently large time. On the one hand, we can compare the second and the fourth term due to the fact that, by definition, $P_2(t) \lesssim C$:
\begin{align*}
    t P_2(t)  \log(P(t))^{\frac{3}{2}}C^{\frac{1}{2}} \lesssim t P_2(t)^{\frac{1}{2}}\log(P(t))^{\frac{3}{2}}C .
\end{align*}
On the other hand, recall that in the estimate \eqref{eq:sigma_12b}, it is remarked that $B \lesssim P(t)$ for the utility of that estimate. Thus, we can compare the first and third terms by using $C \lesssim P(t)$. Then, 
\begin{align*}
    t P_2(t)^{\frac{3}{4}} P(t)^{\frac{5}{4}} C^{-\frac{3}{4}} \lesssim t P_2(t) P(t)^{\frac{3}{2}} C^{-1}.
\end{align*}
Optimizing the sum of the second and third term in \eqref{totalboundKS2}
 yields
 \begin{align*}
  C &= P_2(t)^{\frac{1}{4}} P(t)^{\frac{3}{4}}\log(P(t))^{-\frac{3}{4}}.
 \end{align*}
 Thus, by this choice of $C$,
 \begin{align}\label{twotermsKS2bound}
 t P_{2}(t) P(t)^{\frac{3}{2}} C^{-1} +  t P_2(t)^{\frac{1}{2}}   \log(P(t))^{\frac{3}{2}}C  &= 2  t P_2(t)^{\frac{3}{4}} P(t)^{\frac{3}{4}}  \log(P(t))^{\frac{3}{4}}.
 \end{align}
 Therefore,
\begin{align*}
    h(B,C) &\lesssim t P_2(t)^{\frac{3}{4}} P(t)^{\frac{3}{4}}  \log(P(t))^{\frac{3}{4}},
\end{align*}
which is smaller in large time than the lower bound found in \eqref{lowerboundhBC} because $P_{2}(t)\log(P(t)) \leq P(t)$ in large time by \eqref{ratiolimit}. This concludes the proof for minimizing $h(B,C)$ under the condition $4B \leq C$ and we can state
\begin{align}
K_{S,2} &\leq \mathcal{C} t P_2(t)^{\frac{3}{4}} P(t)^{\frac{3}{4}}  \log(P(t))^{\frac{3}{4}}
\end{align}
for some constant $\mathcal{C}$.

We can check that the choice of $C$ above yields angles
\begin{align*}
     \frac{1}{2}\alpha = \frac{1}{4}\beta = \gamma = \frac{P_{2}(t)}{C(t)} = P_{2}(t)^{\frac{3}{4}}P(t)^{-\frac{3}{4}} \log(P(t))^{\frac{3}{4}} 
 \end{align*}
 which tend to $0$ as $t\rightarrow \infty$ given that $w < 1$ in \eqref{P}.

\section{Closing the estimates}

For $t > T$ for fixed $T$ large enough, we gather the final bounds for $K_T$, $K_{S,1}$ and $K_{S,2}$ in \eqref{KTfinalbound}, \eqref{KS1finalbound} and \eqref{KS2finalbound} to bound the total forcing term $K$. We choose $\mu = \frac{1}{2}$ in \eqref{KS1finalbound} to find, for some new constant $\mathcal{C}$,
\begin{align*}
|K(t,x)| &\leq c_{0}(1+ K_{T}(t,x) + K_{S,1}(t,x) + K_{S,2}(t,x))\\
&\leq \mathcal{C}  t P_{2}(t)\log(P(t))\|K\|_{L^{\infty}([0,t]\times\mathbb{R}^{2})}^{\frac{1}{2}} + \mathcal{C} t\sqrt{\log(t)} P_{2}(t)\log(P(t))
+ K_{T}(t,x)  + K_{S,2}(t,x).
\end{align*}
Since $P(t)$ and $P_{2}(t)$ are increasing functions of time,
\begin{align*}
\|K\|_{L^{\infty}([0,t]\times\mathbb{R}^{2})} 
&\leq \mathcal{C} t P_{2}(t)\log(P(t))\|K\|_{L^{\infty}([0,t]\times\mathbb{R}^{2})}^{\frac{1}{2}} + \mathcal{C} t\sqrt{\log(t)} P_{2}(t)\log(P(t))  \\
& \quad + \|K_{T} \|_{L^{\infty}([0,t]\times\mathbb{R}^{2})} + \|K_{S,2}\|_{L^{\infty}([0,t]\times\mathbb{R}^{2})}.
\end{align*}
By Young's product inequality,
\begin{align*}
    \mathcal{C} t P_{2}(t)\log(P(t))\|K\|_{L^{\infty}([0,t]\times\mathbb{R}^{2})}^{\frac{1}{2}} &\leq \frac{1}{2}\left( \mathcal{C} t P_{2}(t)\log(P(t))\right)^{2} + \frac{1}{2}\|K\|_{L^{\infty}([0,t]\times\mathbb{R}^{2})}.
\end{align*}
Therefore,
\begin{align*}
\|K\|_{L^{\infty}([0,t]\times\mathbb{R}^{2})} 
&\leq \left( \mathcal{C} t P_{2}(t)\log(P(t))\right)^{2} + 2\|K_{T}\|_{L^{\infty}([0,t]\times\mathbb{R}^{2})} + 2\|K_{S,2}\|_{L^{\infty}([0,t]\times\mathbb{R}^{2})}.
\end{align*}

From \eqref{KTfinalbound} and \eqref{KS2finalbound}, it follows that 
\begin{align}
    \|K\|_{L^{\infty}([0,t]\times\mathbb{R}^{2})} 
&\leq  \mathcal{C}' \left( t^2 P_{2}(t)^2\log(P(t))^{2} + t P_2(t) \log(P(t))+  t P_2(t)^{\frac{3}{4}} P(t)^{\frac{3}{4}}  \log(P(t))^{\frac{3}{4}} \right)
\end{align}
for a new constant $\mathcal{C'}$. Recalling \eqref{Ptbound} we then find
\begin{align}
  P(t) &\leq P_2(t)^{\frac{1}{w}} + P(0) + \mathcal{C}' \int_0^t   \left( t^2 P_{2}(t)^2\log(P(t))^{2} + t P_2(t)^{\frac{3}{4}} P(t)^{\frac{3}{4}}  \log(P(t))^{\frac{3}{4}} \right) ds \nonumber \\
  & \leq P_2(t)^{\frac{1}{w}} + P(0) + \mathcal{C}'   \left( t^3 P_{2}(t)^2\log(P(t))^{2} + t^2 P_2(t)^{\frac{3}{4}} P(t)^{\frac{3}{4}}  \log(P(t))^{\frac{3}{4}} \right) \nonumber \\
  &\leq P_2(t)^{\frac{1}{w}} + P(0) + \mathcal{C}' \left( t^3 P_{2}(t)^2\log(P(t))^{2} + \frac{1}{p}\left( \varepsilon P(t)^{\frac{3}{4}} \right)^p + \frac{1}{q} \left( \frac{1}{\varepsilon} t^2 P_2(t)^{\frac{3}{4}}   \log(P(t))^{\frac{3}{4}} \right)^q \right),
\end{align}
where we used Young's inequality for products in the last bound. Setting $w = \frac{1}{3}$, $p = \frac{4}{3}$ and $q =4$ and due to the freedom of choice of $\varepsilon$ (to be set small enough depending on $\mathcal{C}'$ and $p$), we get
\begin{align}\label{Ptalmostfinalbound}
  P(t) &\leq  c t^8 P_2(t)^3 \log(P(t))^3.
\end{align}
Taking the logarithm of both sides and using $\log(x) \leq \frac{x}{6}$ for $x$ large enough, we have
\begin{align*}
  \log(P(t)) &\leq  \log(c) + \log( t^8 P_2(t)^3) + \log( \log(P(t))^3)\\
  &\leq \log(c) + 8\log(tP_2(t)) + \frac{1}{2} \log(P(t)).
\end{align*}
Thus, $ \log(P(t)) \lesssim \log(tP_2(t)) $ and we conclude that
\begin{align*}
   P(t) &\leq  c t^8 P_2(t)^3  \log(tP_2(t))^{3}  
\end{align*}
for large enough constant $c$.

\printbibliography

\end{document}